\documentclass[preprint,12pt]{elsarticle}

\usepackage{amssymb,amsmath,amsthm}
\usepackage{amsfonts}
\usepackage{bbm}
\usepackage{etoolbox}
\usepackage{multirow}
\usepackage{enumerate}
\usepackage{enumitem}
\usepackage{url}

\usepackage{tikz}
\usepackage{tkz-graph}
\usepackage{subfig}
                                                                          
\makeatletter
\def\@author#1{\g@addto@macro\elsauthors{\normalsize%

    \def\baselinestretch{1}%
    \upshape\authorsep#1\unskip\textsuperscript{%
      \ifx\@fnmark\@empty\else\unskip\sep\@fnmark\let\sep=,\fi
      \ifx\@corref\@empty\else\unskip\sep\@corref\let\sep=,\fi
      }%
    \def\authorsep{\unskip,\space}%
    \global\let\@fnmark\@empty
    \global\let\@corref\@empty  
    \global\let\sep\@empty}%
    \@eadauthor={#1}
}
\patchcmd{\ps@pprintTitle}{\footnotesize\itshape
       Preprint submitted to \ifx\@journal\@empty Elsevier
       \else\@journal\fi\hfill\today}{\relax}{}{}
\makeatother

\theoremstyle{plain}

\newtheorem{theorem}{Theorem}

\newtheorem{lemma}[theorem]{Lemma}
\newtheorem{proposition}[theorem]{Proposition}

\newtheorem{remark}[theorem]{Remark}
\newtheorem{question}[theorem]{Question}
\newtheorem{corollary}[theorem]{Corollary}

\newproof{pop}{Proof of Proposition \ref{counterq2}}
\newproof{pop1}{Proof of Proposition \ref{parity}}

\begin{document}    
    \title{On the spectral reconstruction problem for digraphs}
    
    \author[A]{Edward Bankoussou-mabiala}
    \ead{bankoussoumabiala@yahoo.fr}
    
    \author[A]{Abderrahim Boussa\"{\i}ri\corref{cor1}}  
    \cortext[cor1]{Corresponding author}
    \ead{aboussairi@hotmail.com}
    
    \author[A]{Abdelhak Cha\"{\i}cha\^{a}}
    \ead{abdelchaichaa@gmail.com}
    
    \author[A]{Brahim Chergui}
    \ead{cherguibrahim@gmail.com}
    
    \author[A]{Soufiane Lakhlifi}
    \ead{s.lakhlifi1@gmail.com}

\address[A]{Facult\'e des Sciences A\"in Chock,  D\'epartement de Math\'ematiques et Informatique, Laboratoire de Topologie, Alg\`ebre, G\'eom\'etrie et Math\'ematiques Discr\`etes, Universit\'e Hassan II

Km 8 route d'El Jadida,
BP 5366 Maarif, Casablanca, Maroc}
        \begin{frontmatter}
        \begin{abstract}	
			The idiosyncratic polynomial of a graph $G$ with adjacency matrix $A$ is the characteristic polynomial of the matrix $ A + y(J-A-I)$, where $I$ is the 
			identity matrix and $J$ is the all-ones matrix. It follows from a theorem of Hagos (2000) combined with an earlier result of Johnson and Newman (1980) that the idiosyncratic polynomial of a graph is reconstructible from the multiset of the idiosyncratic polynomial of its vertex-deleted subgraphs.
			For a digraph $G$ with adjacency matrix $A$, we define its idiosyncratic polynomial as the characteristic polynomial of the matrix $ A + y(J-A-I)+zA^{T}$.
			By forbidding two fixed digraphs on three vertices as induced subdigraphs, we prove
			that the idiosyncratic polynomial of a digraph is reconstructible from the multiset of the idiosyncratic polynomial of
			its induced subdigraphs on three vertices. As an immediate consequence, the
			idiosyncratic polynomial of a tournament is reconstructible from the collection of its
			$3$-cycles. Another consequence is that all the transitive orientations of a comparability graph
			have the same idiosyncratic polynomial. 
        
        \end{abstract}
        \begin{keyword}
            Digraph, reconstruction problem, idiosyncratic polynomial, hemimorphy, isomorphy, module.
        \end{keyword}

     \end{frontmatter}

\section{Introduction}

	
	Given a graph $G$, the subgraph obtained from $G$ by deleting a vertex $v$ and
	all its incident edges is called a \emph{vertex-deleted} \emph{subgraph}. The
	multiset of vertex-deleted subgraphs, given up to isomorphism, is called the
	\emph{deck} of $G$. We say that $G$ is \emph{reconstructible} if it is
	uniquely determined (up to isomorphism) by its deck. The well-known Graph
	Reconstruction Conjecture of Kelly \cite{kelly1957congruence} and Ulam
	\cite{ulam1960collection} states that all finite graphs on at least three
	vertices are reconstructible. A problem which is closely related to this conjecture is
	the reconstruction of graph invariant polynomials. We mean by a \emph{graph
	invariant} a function $\mathcal{I}$ from the set of all graphs into any
	commutative ring such that $\mathcal{I}(G)=\mathcal{I}(H)$ if $G$ and $H$ are
	two isomophic graphs. We say that a graph invariant is \emph{reconstructible}
	if it is uniquely determined by the deck. For example, Tutte
	\cite{tutte1979all} proved that the characteristic polynomial and the
	chromatic polynomial are reconstructible. A natural question is to ask if a
	graph invariant polynomial can be reconstructed from the polynomial deck, that
	is, from the multiset of the polynomials of the vertex-deleted subgraphs? For
	the characteristic polynomial the problem is still open. It was posed by
	Cvetkovic at the XVIII International Scientific Colloquium in Ilmenau in 1973.
	Hagos \cite{hagos2000characteristic} proved that the characteristic polynomial
	of a graph is reconstructible from its polynomial deck together
	with the polynomial deck of its complement. The \emph{idiosyncratic
	polynomial} of a graph $G$ with adjacency matrix $A$ is the characteristic polynomial of the matrix
	obtained by replacing each non-diagonal zero in $A$
	with an indeterminate $x$, that is, the characteristic polynomial of the matrix $ A + x(J-A-I)$. 
	Johnson and Newman \cite{johnson1980note} consider a slightly different polynomial which can be viewed as the idiosyncratic polynomial of the complement of $G$. It follows from their main theorem that two graphs have the same idiosyncratic polynomial
	if only if they are cospectral, and their complements are also cospectral.
	Then by Hagos' theorem, the idiosyncratic polynomial of a graph
	$G$ is recontructible from its idiosyncratic polynomial deck.
	
	The reconstruction conjecture was also considered for tournaments and more
	generally for digraphs. In this area,
    Stockmeyer \cite{stockmeyer1977falsity} construct for every positive integer $n$ two
	non isomorphic tournaments $B_{n}$ and $C_{n}$ on the same vertex set
	$\left\{  0,\ldots,2^{n}\right\}  $. For this he consider the tournament
	$A_{n}$ defined on $\left\{  1,\ldots,2^{n}\right\}  $ by $(i,j)$ is an arc of
	$A_{n}$ if only if $odd(j-i)\equiv1\pmod{4} $, where $odd(x)$ is the largest odd
	divisor of $x$. The tournaments
	$B_{n}$ and $C_{n}$ are obtained from $A_{n}$ by adding the vertex $0$. In the
	tournament $B_{n}$, the vertex $0$ dominates $2,4\ldots,2^{n}$ and is
	dominated by $1,3\ldots,2^{n}-1$. In the tournament $C_{n}$, the vertex $0$
	dominates $1,3\ldots,2^{n}-1$ and is dominated by $2,4\ldots,2^{n}$. It is
	proved in \cite{stockmeyer1977falsity} that for $1\leq k\leq2^{n}$, the
	tournaments $B_{n}-$ $k$ and $C_{n}-$ $(2^{n}+1-k)$ are isomorphic. Then the
	pair $B_{n}$ and $C_{n}$ form a counterexample for the reconstruction
	conjecture. As mentioned by Pouzet \cite{pouzet1979note}, Dumont checked
	that for $n\leq6$ the difference (in absolute value) between the determinants
	of $B_{n}$ and $C_{n}$ is $1$. This fact is perhaps true for arbitrary $n$ but
	we are not able to prove it. However, we have the following result.
	
	\begin{proposition}
		\label{parity}For $n\geq3$, the determinants of $B_{n}$ and $C_{n}$ do not
		have the same parity.
	\end{proposition}
	
	Fra\"{\i}ss\'{e} \cite{fraisse1970abritement} considered a strengthening of the
	reconstruction conjecture for the class of relations which contains graphs
	and digraphs. For digraphs, Fra\"{\i}ss\'{e}'s problem can be stated as
	follow. Let $G$ and $H$ be two digraphs with the same vertex set $V$ and assume
	that for every proper subset $W$ of $V$, the subdigraphs $G\left[  W\right]  $
	and $H\left[  W\right]  $, induced by $W$ are isomorphic. Is it true that $G$
	and $H$ are isomorphic? Lopez \cite{lopez1978indeformabilite} proved that the
	answer is positive when $\left\vert V\right\vert \geq7$. It follows that if
	$G\left[  W\right]  $ and $H\left[  W\right]  $ are isomorphic for every
	subset $W$ of size at most $6$, then $G$ and $H$ are isomorphic.
	Motivated by Lopez's theorem, 
	we can ask the following question.
	 
	\begin{question}\label{invariant}
		Let $\mathcal{I}$ be a digraph invariant polynomial and let $G$ be a digraph. Is the polynomial $\mathcal{I}(G)$ reconstructible from the 
		collection $\{\mathcal{I}(H): H\,\in \mathcal{H}\}$, where
		$\mathcal{H}$ is the set of proper induced subdigraphs of $G$?
	\end{question}
	 
	 In this paper, we will address this question for idiosyncratic polynomial extended to digraphs and defined as follow.
	 Let $G$ be a digraph with adjacency matrix $A$. The \emph{generalized adjacency matrix} of $G$ is $A(y,z)=A + y(J-A-I)+zA^{T}$. The idiosyncratic polynomial of $G$ as 
	 the characteristic polynomial of $A(y,z)$. The presence of
	 $zA^{T}$ comes from the fact that the adjacency matrix of a digraph is not necessarily symmetric.
	 It is not difficult to see that if two digraphs have the same idiosyncratic polynomial then they have the same characteristic polynomial, moreover their complement and their converse are also the same characteristic polynomial.
	  
	We prove that Question \ref{invariant} is not true for arbitrary digraphs.
	Our counterexamples are borrowed from \cite{boussairi2004c3} where they have been used
	in another context. All of these counterexamples contain one of two particular digraphs
	called \emph{flag}. Following \cite{boussairi2004c3} a \emph{flag}  is a digraph with
	 vertex set $\{u,v,w\}$ and whose arcs set is either $\left\{  \left(  u,v\right)
	,\left(  u,w\right)  ,\left(  w,u\right)  \right\}  $ or $\left\{  \left(
	v,u\right)  ,\left(  u,w\right)  ,\left(  w,u\right)  \right\}  $. A
	\emph{flag-free} \emph{digraph} is a digraph in which there is no flag as
	induced subdigraph. 
	
	Our main result is stated as follow.
	
	\begin{theorem}
	\label{maintheorem} Let $G$ and $H$ be two flag-free digraphs with the same
	vertex set $V$ of size at least $5$. If for every $3$-subset $W$ of $V$, the
	induced subdigraphs $G\left[  W\right]  $ and $H\left[  W\right]  $ have the
	same  idiosyncratic polynomial, then $G$ and $H$ have the same  idiosyncratic polynomial.
	\end{theorem}
	
	As an application, we obtain the following corollary about tournaments.
	
	\begin{corollary}
	\label{maincas1}Two tournaments with the same $3$-cycles have the same
	 idiosyncratic polynomial.
	\end{corollary}
	
	Posets form an important class of digraphs for which the reconstruction
	problem is still open. Ille and Rampon \cite{ille1998reconstruction} proved
	that a poset is reconstructible by its deck together with its comparability graph.
	
	Following Habib
	\cite{habib1984comparability}, a parameter of a poset is said to be
	\emph{comparability invariant} if all posets with a given comparability graph
	have the same value of that parameter. The dimension and the number of
	transitive extension of a poset are two examples of comparability invariants.
	
	The next corollary is another consequence of Theorem \ref{maintheorem}.
	\begin{corollary}
	\label{maincas2} All the transitive orientations of a comparability graph have the same 
	idiosyncratic polynomial.
	\end{corollary}
	
	\section{Preliminaries}
	
	A \emph{graph} $G$ consists of a finite set $V$ of \emph{vertices} together
	with a set $E$ of unordered pairs of distinct vertices of $V$ called
	\emph{edges}. Let $G=(V,E)$ be a graph. With respect to an ordering
	$v_{1},\ldots,v_{n}$\ of $V$, the \emph{adjacency matrix} of $G$ is the
	$n\times n$ zero diagonal matrix $A=[a_{ij}]$\ in which $a_{ij}=1$ if $\left(
	v_{i},v_{j}\right)  \in E$ and $0$ otherwise. The \emph{complement }of a graph
	$G=(V,E)$ is the graph $\overline{G}$ with the same vertices as $G$ and such
	that, for any $u,v\in V$, $\left\{  u,v\right\}  $ is an edge of $\overline
	{G}$ if and only if $\left\{  u,v\right\}  \notin E$.
	
	A \emph{directed graph} or \emph{digraph} $G$ is a pair $(V,E)$ where $V$ is a
	nonempty set $V$ of \emph{vertices} \ and $E$ is a set of ordered pairs of
	distinct vertices called \emph{arcs}. Let $W$ be a subset of $V$ the
	\emph{subdigraph} of $G$ \emph{induced} by $W$ is the digraph $G\left[
	W\right]  $ whose vertex set is $W$ and whose arc set consists of all arcs of
	$G$ which have end-vertices in $W$. 
	A digraph $G=(V,E)$ is \emph{symmetric} if, whenever
	$(u,v) \in E$ then $(v,u)$ $\in E$. There is a natural one
	to one correspondence between graphs and symmetric digraphs.
	
	Let $G=(V,E)$ be a digraph. With respect to an ordering $v_{1},\ldots,v_{n}%
	$\ of $V$, the \emph{adjacency matrix} of $G$ is the $n\times n$ zero diagonal
	matrix $A=[a_{ij}]$ \ in which $a_{ij}=1$ if $\left(  v_{i},v_{j}\right)  \in
	E$ and $0$ otherwise. The \emph{converse }of $G$, denoted by $G^{\ast}$, is
	the digraph obtained from $G$ by reversing the direction of all its arcs. The
	adjacency matrix of $G^{\ast}$ is the transpose $A^{T}$ of the matrix $A$, in
	particular $P_{G}\left(  X\right)  =P_{G^{\ast}}\left(  X\right)  $. The
	\emph{complement }of $G$ is the digraph $\overline{G}$ with vertex set $V$ and
	such that, for any $u,v\in V$, $\left(  u,v\right)  $ is an arc of
	$\overline{G}$ if and only if $\left(  u,v\right)  \notin E$. The adjacency
	matrix of $\overline{G}$ is $J-A-I$.
	
	 An
	\emph{oriented} \emph{graph} is a digraph $G=\left(  V,E\right)  $ such that
	for $x,y\in V$ , if $\left(  x,y\right)  \in$ $E$, then $\left(  y,x\right)
	\notin$ $E$. Let $G$ be a graph. An \emph{orientation} of $G$ is an assignment of a direction to 
	each edge of $G$ so that we obtain an oriented graph.
	 A \emph{tournament} is an orientation of the complete graph. An oriented graph
	  is a \emph{poset} if, whenever $\left(  x,y\right)$ and 
	  $\left(  y,z\right)$ are arcs then $\left(  x,z\right)$ is also an arc.
	  A \emph{transitive orientation} of a graph is one where the resulting oriented graph is
	  a poset. \emph{Comparability graphs} are the class of graphs that have a transitive orientation.

	\section{Determinant of Stockmeyer's tournaments}\label{Stockmeyer}
	
	In this section, we prove Proposition \ref{parity}. For this, we will use the following lemma, which 
	is a particular case of \cite[Equality~(A)]{pouzet1979note}.
    \begin{lemma}
    	For a pair $\left(  G,H\right)  $ of digraphs, satisfying the hypothesis
    	 of the reconstruction Conjecture, we have
    	\begin{equation}\label{a}
    	\det(G)-\det(H)=\left(  -1\right)  ^{n+1}\left[  C(G)-C(H)\right] %
    	\end{equation}
    	where $C(G)$ and $C(H)$ are respectively the numbers of Hamiltonian cycles of
    	$G$ and $H$.
    \end{lemma}
	Remark that in this Lemma, Equality (\ref{a}) is slightly different 
	from Equality (A) of \cite{pouzet1979note}. The reason is that, we do not use the same definition of cycle. In our paper, we mean by a (directed)
	\emph{cycle} of a digraph $G$ every subdigraph with vertex set $\left\{
	x_{1},\ldots,x_{t}\right\}  $ and arcs set $\left\{  x_{1}x_{2},\ldots
	,x_{t-1}x_{t},x_{t}x_{1}\right\}  $. This cycle is said to be
	\emph{Hamiltonian} if it goes through each vertex of $G$ exactly once. A path
	 of a digraph $G$ is a subdigraph with vertex set $\left\{  x_{1},\ldots
	,x_{t}\right\}  $ and arc set $\left\{  x_{1}x_{2},\ldots,x_{t-1}%
	x_{t}\right\}  $. Such path is denote by $x_{1}x_{2}\ldots x_{t}$. The notion of Hamiltonian path is defined similarly.
	
	Let $T$ be a tournament and let $v$ be a vertex of $T$. We denote by $N^{+}(v)$
	(resp.$N^{-}(v)$) the \emph{out-neighborhood} (resp. the\emph{
		in-neighborhood}) that is the set of vertices dominated by $v$ (resp. that
	dominate $v$).
	
	\begin{remark}
		\label{path}There is the natural one-to-one correspondence between Hamiltonian
		cycles of $T$ and Hamiltonian paths of $T-v$ from a vertex $x\in N^{+}(v)$
		to a vertex of $N^{-}(v)$.
	\end{remark}

	\begin{pop1}
		 Let $\mathcal{O}=\left\{  1,3,\ldots
		,2^{n}-1\right\}  $ and $\mathcal{E}=\left\{  2,4,\ldots,2^{n}\right\}  $. The
		set $\mathcal{P}$ of Hamiltonian paths of $A_{n}$ is partitioned into four subsets:
		
		\begin{description}
			\item[i)] $\mathcal{P}_{o,o}$ the set of Hamiltonian paths joining two vertices in
			$\mathcal{O}$.
			
			\item[ii)] $\mathcal{P}_{e,e}$ the set of Hamiltonian paths joining two vertices in
			$\mathcal{E}$.
			
			\item[iii)] $\mathcal{P}_{o,e}$ the set of Hamiltonian paths joining a vertex in
			$\mathcal{O}$ to a vertex in $\mathcal{E}$.
			
			\item[iv)] $\mathcal{P}_{e,o}$ the set of Hamiltonian paths joining a vertex in
			$\mathcal{E}$ to a vertex in $\mathcal{O}$.
			
			 We will prove that $\left\vert \mathcal{P}_{o,o}\right\vert =\left\vert\mathcal{P}_{e,e}\right\vert $. 
			Let $x_{1}x_{2}\ldots x_{2^{n}}%
			\in\mathcal{P}_{o,o}$. For $i=1,\ldots,2^{n}$, we set $\widetilde{x}
			_{i}:=2^{n}-x_{2^{n}-i+1}+1$. 
		It is easy to see that
		\end{description}	
		\[
		odd(\widetilde{x}_{i+1}-\widetilde{x}_{i})=odd(x_{(2^{n}-i)+1}-x_{2^{n}-i})=1
		\]
		 Moreover, $\widetilde{x}_{1}$, $\widetilde{x}_{2^{n}}\in$
		 $\mathcal{E}$, then $\widetilde{x}_{1}\widetilde{x}_{2}\ldots\widetilde
		 {x}_{2^{n}}\in\mathcal{P}_{e,e}$. It follows that there is a one-to-one
		correspondence between $\mathcal{P}_{o,o}$ and $\mathcal{P}_{e,e}$. To
		conclude it suffices to apply Equality (\ref{a}) and Redei's theorem \cite{redei1934kombinatorischer} asserting that the number of Hamiltonian paths in a
		tournament is always odd. 
	\end{pop1}

		\section{Counterexample for Question \ref{invariant}}
	
Consider the digraph $G$  with vertex set $\left\{  1,\ldots,n\right\}  $ and whose arcs are
	 $  (1,2),\left(  n-1,n\right) $, $\left(  i,i+1\right)$ and $\left(
	1,i+1\right)$  for $i=1,\ldots,n-2 $. Let $G^{\prime}$ be the digraph
	obtained from $G$ by reversing the arc $\left(  n-1,n\right) $. These two
	 digraphs are drawn in Figure $1$.
	
	 We will prove the following proposition.
	 
	 \begin{proposition}\label{counterq2}
	 	Let $G$ and $G^{\prime}$ be the digraphs defined above. Then we have
	 \begin{description}
	 	\item[i)] $G$ and $G^{\prime}$ do not have the same idiosyncratic polynomial;
	 	\item[ii)] $G\left[  W\right]  $ and $G^{\prime}\left[  W\right]  $
	 	have the same  idiosyncratic polynomial for every proper subset $W$ of
	 	$\left\{  1,\ldots,n\right\}$.
	 \end{description}	
	 \end{proposition}
   Our proof is based on the Coates determinant formula \cite{coates1959flow}. This can be used to evaluate the determinant of
   the adjacency matrix of a digraph. Let $H$
   be a digraph on $n$ vertices (possibly with loops). A linear subdigraph $L$ of
   $H$ is a vertex disjoint union of some cycles in $H$ (we consider a loop as a
   cycle of length $1$). The set of linear subdigraphs of $H$ with $n$ vertices
   is denote by $\mathcal{L}(H)$. If $N$ is the adjacency matrix of $H$,
   then from Coates determinant formula we have
   \begin{equation}
   \det\left(  N\right)  =\left(  -1\right)  ^{n}\underset{L\in\mathcal{L}%
   	(H)}{\sum^{n}}\left(  -1\right)  ^{\left\vert L\right\vert } \label{b}%
   \end{equation}
   where $\left\vert L\right\vert $ is the number of cycles in $L$.

    \begin{pop}
    	To prove that the digraphs $G$ and $G^{\prime}$ do not have the same idiosyncratic polynomial, it suffices 
    	to check that their complement $\overline{G}$ and
	$\overline{G^{\prime}}$ do not have the same determinant. Let $A$ and
	$A^{\prime}$ be the adjacency matrices of $G$ and $G^{\prime}$ respectively. Then, the adjacency
	matrices of $\overline{G}$ and $\overline{G^{\prime}}$ are respectively
	$\overline{A}=J-A-I$ and $\overline{A^{\prime}}=J-A^{\prime}-I$.
	
	Let \[
	\widetilde{A}:=\left(
	\begin{tabular}
	[c]{c|c}%
	$A+I$ & $\mathbbm{1}$\\\hline
	$\mathbbm{1}^{t}$ & $1$%
	\end{tabular}
	\right)\mbox{, }\widetilde{A'}:=\left(
	\begin{tabular}
	[c]{c|c}%
	$A^{\prime}+I$ & $\mathbbm{1}$\\\hline
	$\mathbbm{1}^{t}$ & $1$%
	\end{tabular}
	\right)
	\]
	where $\mathbbm{1}$ the all one column vector of dimension $n$.
	It is easy to
	see that%
	\[
	\det(\overline{A})=\left(  -1\right)  ^{n}\det
	(
	\widetilde{A}
	)\mbox{, }\det(\overline{A^{\prime}})=\left(  -1\right)
	^{n}\det(\widetilde{A'})
	\]
	
	Remark that
	 $\widetilde{A}$ and $\widetilde{A'}$ can be viewed as the
	adjacency matrices of the digraphs $\widetilde{G}$ and $\widetilde{G^{\prime}}$
	defined on the set $\left\{
	1,\ldots,n+1\right\}  $ as follow.
	The arcs of $\widetilde{G}$ are  $(1,2)$, $\left(
	n-1,n\right)$, $\left(  i,i+1\right)$, $\left(  1,i+1\right)$ for $i=1,\ldots,n-2$
	and $(i,i)$, $(i,n+1)$, $(n+1,i)$ for $i=1,\ldots
	,n$. The digraph $\widetilde{G^{\prime}}$ is obtained from
	$\widetilde{G}$ by reversing the arc $\left(  n-1,n\right) $.

	We will evaluate $\det(\widetilde{A})-\det(\widetilde{A'})$ by using the
	formula (\ref{b}).
	For this,
	we partition
	$\mathcal{L}(\widetilde{G})$ into four subsets:
	
	\begin{itemize}
		\item $\mathcal{L}_{1}(\widetilde{G})$ the set of linear subdigraphs containing
		the arcs $(1,2)$ and $\left(  n-1,n\right)$.
		
		\item $\mathcal{L}_{2}(\widetilde{G})$ the set of linear subdigraphs containing
		the arc $(1,2)$ but not $\left(  n-1,n\right)$.
		
		\item $\mathcal{L}_{3}(\widetilde{G})$ the set of linear subdigraphs containing
		the arc $\left(  n-1,n\right)  $ but not $(1,2)$.
		
		\item $\mathcal{L}_{4}(\widetilde{G})$ the set of linear subdigraph containing
		neither the arc $\left(  n-1,n\right)  $ nor $(1,2)$.
	\end{itemize}
	
	We define a similar partition of $\mathcal{L}(\widetilde{G^{\prime}})$ by
	replacing the arc $\left(  n-1,n\right)  $ by $\left(  n-1,n\right) $.
	Clearly we have $\mathcal{L}_{2}(\widetilde{G})=\mathcal{L}_{2}(\widetilde
	{G^{\prime}})$, $\mathcal{L}_{4}(\widetilde{G})=\mathcal{L}_{4}(\widetilde
	{G^{\prime}})$ and $\mathcal{L}_{3}(\widetilde{G^{\prime}})=\left\{  L^{\ast
	}:L\in\mathcal{L}_{3}(\widetilde{G})\right\}  $. Moreover, $\mathcal{L}_{1}(\widetilde
{G^{\prime}})$ is empty and $\mathcal{L}%
	_{1}(\widetilde{G})$ contains only the Hamiltonian cycle whose arcs are
	$(i,i+1)$ for $i=1,\ldots,n$ and $(n+1,1)$.
	
	Using formula (\ref{b}), we get $\det(\widetilde{A})-\det(\widetilde{A^{\prime}%
	})=\left(  -1\right)  ^{n+1}$. Hence $\det(\overline{A})-\det(\overline
	{A^{\prime}})=-1$. It follows that $G$ and $G^{\prime}$ do not have the same
	 idiosyncratic polynomial.
	
	We will prove now that $G\left[  W\right]  $ and $G^{\prime}\left[  W\right]  $
	have the same  idiosyncratic polynomial for every proper subset $W$ of
	$\left\{  1,\ldots,n\right\}  $. This is true when $\left\{
	1,2,n,n-1\right\}  $ is not entirely contained in $W$, because $G\left[
	W\right]  =G^{\prime}\left[  W\right]  $ or $G^{\ast}\left[  W\right]
	=G^{\prime}\left[  W\right]  $. So we can assume that $\left\{  1,2,n,n-1\right\}
	\subseteq W$. Let $k\in\left\{  3,\ldots,n-2\right\}\setminus W$. The set
	$W$ is partitioned into two nonempty subsets $W_{1}:=W\cap\left\{
	1,\ldots,k-1\right\}  $ and $W_{2}:=W\cap\left\{  k+1,\ldots,n\right\}  $.
	Clearly, there is no arc between $W_{1}$ and $W_{2}$ in $G\left[  W\right]  $
	and $G^{\prime}\left[  W\right]  $. Moreover $G\left[  W_{1}\right]
	=G^{\prime}\left[  W_{1}\right]  $ and $G^{\ast}\left[  W_{2}\right]
	=G^{\prime}\left[  W_{2}\right]  $. Then the generalized adjacency matrices of
	$G\left[  W\right]  $ and $G^{\prime}\left[  W\right]  $ have the form%
	
	\[
	A=\left(
	\begin{array}
	[c]{cc}%
	A_{11} & \alpha\beta^{t}\\
	\beta\alpha^{t} & A_{22}%
	\end{array}
	\right)  \text{ \ \ \ and \ \ \ }B=\left(
	\begin{array}
	[c]{cc}%
	A_{11} & \alpha\beta^{t}\\
	\beta\alpha^{t} & A_{22}^{t}%
	\end{array}
	\right)  \text{,}%
	\]

	where $\alpha=\left(
	\begin{array}
	[c]{c}%
	1\\
	\vdots\\
	1
	\end{array}
	\right)  $ and $\beta=\left(
	\begin{array}
	[c]{c}%
	y\\
	\vdots\\
	y
	\end{array}
	\right)  $.
	We conclude by Proposition \ref{HLowey} below.
  \end{pop}

     \begin{proposition}
     	\label{HLowey}Suppose $k$, $n$ are positive integers such that $k<n$. Suppose
     	that $\alpha$, $\gamma\in F^{k}$, $\beta\in F^{n-k}$, $A_{11}\in F^{k,k}$,
     	$A_{22}\in F^{n-k,n-k}$. Let
     	
     	$A=\left(
     	\begin{array}
     	[c]{cc}%
     	A_{11} & \alpha\beta^{t}\\
     	\beta\gamma^{t} & A_{22}%
     	\end{array}
     	\right)  $, $B=\left(
     	\begin{array}
     	[c]{cc}%
     	A_{11} & \alpha\beta^{t}\\
     	\beta\gamma^{t} & A_{22}^{t}%
     	\end{array}
     	\right)  $.
     	
     	Then $A$ and $B$ have the same characteristic polynomial.
     \end{proposition}
	 This proposition is the direct consequence of \cite[Lemma~5]{hartfiel1984matrices}.

	 	\begin{figure}[h]
	 	\begin{center}
	 		\setlength{\unitlength}{0.7cm}
	 		\begin{picture}(30,9.1)
	 		
	 		\linethickness{3mm}

	 		\thicklines
	 		\qbezier (1,2)(2.5,2.25)(4,2)
	 		\qbezier (1,2)(2.5,1.7)(4,2)
	 		\put(1,4){\circle*{.2}}
	 		\put(1,2){\circle*{.2}}
	 		\put(4,2){\circle*{.2}}
	 		\put(0.8,4.2){$0$}
	 		\put(0.8,1.2){$1$}
	 		\put(3.8,1.2){$2$}
	 		\thicklines
	 		\put(2.4,2.125){\vector(1,0){0.26}}
	 		\put(2.6,1.85){\vector(-1,0){0.26}}
	 		\put(1,2){\line(0,1){2}}
	 		\put(1,4){\vector(0,-1){1.2}}
	 		\thicklines
	 		\qbezier (4,2)(5.5,2.25)(7,2)
	 		\qbezier (4,2)(5.5,1.7)(7,2)
	 		\put(7,2){\circle*{.2}}
	 		\put(10,2){\circle*{.2}}
	 		\put(6.8,1.2){$3$}
	 		\put(9.3,1.2){$n-3$}
	 		\thicklines
	 		\put(5.4,2.125){\vector(1,0){0.26}}
	 		\put(5.6,1.85){\vector(-1,0){0.26}}

	 		\put(7.5,2){\circle*{.1}}
	 		\put(8,2){\circle*{.1}} 
	 		\put(9,2){\circle*{.1}}
	 		\put(8.5,2){\circle*{.1}}
	 		\put(8,2){\circle*{.1}} 
	 		\put(9.5,2){\circle*{.1}}

	 		\qbezier (10,2)(11.5,2.25)(13,2)
	 		\qbezier (10,2)(11.5,1.7)(13,2)
	 		\put(13,2){\circle*{.2}}
	 		\put(16,2){\circle*{.2}}
	 		\put(16,4){\circle*{.2}}
	 		\put(12.3,1.2){$n-2$}
	 		\put(15.3,1.2){$n-1$}
	 		\put(15.8,4.2){$n$}
	 		\put(11.4,2.125){\vector(1,0){0.26}}
	 		\put(11.6,1.85){\vector(-1,0){0.26}}
	 		\put(1,2){\line(0,1){2}}
	 		\put(1,4){\vector(0,-1){1.2}}
	 		\qbezier (13,2)(14.5,2.25)(16,2)
	 		\qbezier (13,2)(14.5,1.7)(16,2)
	 		\put(14.4,2.125){\vector(1,0){0.26}}
	 		\put(14.6,1.85){\vector(-1,0){0.26}}
	 		\put(16,2){\line(0,1){2}}
	 		\put(16,2){\vector(0,1){1.2}}
	 		\put(8.3,0.1){$G$}

	 		
	 		\thicklines
	 		\qbezier (1,7)(2.5,7.25)(4,7)
	 		\qbezier (1,7)(2.5,6.7)(4,7)
	 		\put(1,7){\circle*{.2}}
	 		\put(1,9){\circle*{.2}}
	 		\put(4,7){\circle*{.2}}
	 		\put(0.8,9.2){$0$}
	 		\put(0.8,6.2){$1$}
	 		\put(3.8,6.2){$2$}
	 		\thicklines
	 		\put(2.4,7.125){\vector(1,0){0.26}}
	 		\put(2.6,6.85){\vector(-1,0){0.26}}
	 		\put(1,7){\line(0,1){2}}
	 		\put(1,9){\vector(0,-1){1.2}}
	 		
	 		\qbezier (4,7)(5.5,7.25)(7,7)
	 		\qbezier (4,7)(5.5,6.7)(7,7)
	 		\put(7,7){\circle*{.2}}
	 		\put(10,7){\circle*{.2}}
	 		\put(6.8,6.2){$3$}
	 		\put(9.3,6.2){$n-3$}
	 		\put(5.4,7.125){\vector(1,0){0.26}}
	 		\put(5.6,6.85){\vector(-1,0){0.26}}

	 		\put(7.5,7){\circle*{.1}}
	 		\put(8,7){\circle*{.1}} 
	 		\put(9,7){\circle*{.1}}
	 		\put(8.5,7){\circle*{.1}}
	 		\put(8,7){\circle*{.1}} 
	 		\put(9.5,7){\circle*{.1}}

	 		\qbezier (10,7)(11.5,7.25)(13,7)
	 		\qbezier (10,7)(11.5,6.7)(13,7)
	 		\put(13,7){\circle*{.2}}
	 		\put(16,7){\circle*{.2}}
	 		\put(16,9){\circle*{.2}}
	 		\put(12.3,6.2){$n-2$}
	 		\put(15.3,6.2){$n-1$}
	 		\put(15.8,9.2){$n$}
	 		\put(11.4,7.125){\vector(1,0){0.26}}
	 		\put(11.6,6.85){\vector(-1,0){0.26}}
	 		\put(1,2){\line(0,1){2}}
	 		\put(1,4){\vector(0,-1){1.2}}
	 		\qbezier (13,7)(14.5,7.25)(16,7)
	 		\qbezier (13,7)(14.5,6.7)(16,7)
	 		\put(14.4,7.125){\vector(1,0){0.26}}
	 		\put(14.6,6.85){\vector(-1,0){0.26}}
	 		\put(16,7){\line(0,1){2}}
	 		\put(16,9){\vector(0,-1){1.2}}
	 		\put(8.3,5.1){$G^\prime$}

	 		\put(7.5,-1){Figure 1}
	 		
	 		\end{picture}
	 	\end{center}
	 \end{figure}

	\section{Isomorphy and Hemimorphy}
	
	Two digraphs $G=(V,E)$ and $G^{\prime}=(V^{\prime},E^{\prime})$ are said to be
	\emph{isomorphic} if there is a bijection $\varphi$ from $V$ onto $V^{\prime}$
	which preserves arcs, that is $(x,y)\in E$ if and only if $(\varphi
	(x),\varphi(y))\in E^{\prime}$. Any such bijection is called an
	\emph{isomorphism}. We say that $G$ and $G^{\prime}$ are \emph{hemimorphic}%
	,\emph{ }if there exists an isomorphism from $G$ to $G^{\prime}$ or from
	$G^{\ast}$ to $G^{\prime}$.
	
	Let $G$ and $H$ be two digraphs with the same vertex set $V$ of size $n$. Let
	$k\in\left\{  1,\ldots,n\right\}  $, we say that $G$ and $H$ are
	$k$\emph{-hemimorphic }if for every $k$-subset $W$ of $V$, the subdigraphs
	$G\left[  W\right]  $ and $H\left[  W\right]  $ are hemimorphic. More
	generally, let $K$ be a subset of $\{1,\ldots,n\}$, the digraphs $G$ and $H$
	are $K$-\emph{hemimorphic } if they are $k$-hemimorphic for every
	$k\in K$.
	
	Let $G=\left(  V,E\right)  $ be a digraph with at least $5$ vertices. For a
	subset $W$ of $V$ we denot by $\nu\left(  W\right)  $ the number of arcs
	contained in $G\left[  W\right]  $. Let $x_{1}\neq x_{2}\in V$ and let
	$x_{3},x_{4},x_{5}$ be three distinct vertices in $V\setminus\left\{
	x_{1},x_{2}\right\}  $. By applying the inclusion-exclusion principle, we get:  
	$
	6\nu(x_{1},x_{2})=2\nu(
	x_{3},x_{4},x_{5})+\underset{i\in\{3,4,5\}}{%
	{\displaystyle\sum}
	}2\nu(x_{1},x_{2},x_{i})-\underset{i,j\in\{
	3,4,5\}}{%
	{\displaystyle\sum}
	}(\nu(x_{1},x_{i},x_{j})+\nu(x_{2}%
	,x_{i},x_{j}))
	$

	Using this Formula, we deduce the following lemma.
	
	\begin{lemma}
	Two $3$-hemimorphic digraphs with at least five vertices are $2$-hemimorphic.
	\end{lemma}
	
	This result \ can also be obtained by applying the Combinatorial Lemma due to
	Pouzet \cite[Theorem~1]{pouzet1976application}.
	
	Let $G=(V,E)$ be a digraph. Following \cite{ehrenfeucht1999theory}, a subset $W$ of $V$ is a
	\emph{module} of $G$ if for any $a,b\in W$ and $x\in V\setminus W$, $(a,x)\in E$
	(resp. $(x,a)\in E)$) if and only if $(b,x)\in E$ (resp. $(x,b)\in E$). For a
	subset $W$ of $V$, we denote by $Inv(W,G)$ the digraph obtained from $G$ by
	reversing all the arcs of $G\left[  W\right]  $. Clearly, if $W$ is a module of $G$
	then $G$ and $Inv(W,G)$ are $\{2,3\}$-hemimorphic. More generally, if there
	exists a sequence $G_{0}=G,\ldots,G_{m}=H$ of digraphs such that for
	$i=0,\ldots,m-1$, $G_{i+1}=Inv(W_{i},G_{i})$ where $W_{i}$ is a module
    of $G_{i}$
	then $G$ and $H$ are $\{2,3\}$-hemimorphic. The converse is true  for flag-free digraphs
	as shown in the following theorem.
	
	\begin{theorem}
	\textup{\cite{boussairi2004c3}}\label{Biltnew} Let $G$ and $H$ be two flag-free digraphs with the
	same vertex set. If $G$ and $H$ are $\{2,3\}$-hemimorphic then there
	exists a sequence $G_{0}=G,\ldots,G_{m}=H$  of flag-free digraphs such
	that for $i=0,\ldots,m-1$, $G_{i+1}=Inv(W_{i},G_{i})$ where $W_{i}$ is a module of
	$G_{i}$.
	\end{theorem}
   This theorem is not valid for arbitrary digraphs. It suffices to consider the counterexample used in
   Question \ref{invariant}.

	\section{Proof of the Main Theorem}
	The proof of the main theorem is a direct consequence of Proposition \ref{23key}
	and Lemma \ref{3idio} below.

	\begin{lemma}\label{3idio}
		Let $G$ and $H$ be digraphs with $3$ vertices. Then the following assertions
		are equivalent
		
		\begin{description}
			\item[i)] $G$ and $H$ have the same idiosyncratic polynomial.
			
			\item[ii)] $P_{G}\left(  X\right)  =P_{H}\left(  X\right)  $ and
			$P_{\overline{G}}\left(  X\right)  =P_{\overline{H}}\left(  X\right)  $.
			
			\item[iii)] $G$ and $H$ are hemimorphic.
		\end{description}
	\end{lemma}
	
	\begin{proof}
		Up to hemimorphy and complementation there are exactly seven digraphs with three vertices, $G_{1}%
		,\ldots,G_{6}$ and $F$ (see Figure
		2). By simple computation, we have
		
		\begin{description}
			\item[i)] $P_{G_{1}}\left(  X\right)  =X^{3}$ and $P_{\overline{G_{1}}}\left(
			X\right)  =X^{3}-3X-2$.
			
			\item[ii)] $P_{G_{2}}\left(  X\right)  =X^{3}$ and $P_{\overline{G2}}\left(
			X\right)  =X^{3}-2X-1$.
			
				\item[iii)] $P_{G_{3}}\left(  X\right)  =P_{\overline{G_{3}}}\left(  X\right)
			=X^{3}-1$.
			
			\item[iv)] $P_{G_{4}}\left(  X\right)  =P_{\overline{G_{4}}}\left(  X\right)
			=X^{3}$.
			
			\item[v)] $P_{G_{5}}\left(  X\right)  =X^{3}$ and $P_{\overline{G_{5}}%
			}\left(  X\right)  =X^{3}-X$.

			\item[vi)]$P_{G_{6}}\left(  X\right)  =X^{3}$ and $P_{\overline{G_{6}}%
			}\left(  X\right)  =X^{3}-X-1$.

			\item[vii)] $P_{F}\left(  X\right)  =P_{\overline{F}}\left(  X\right)
			=X^{3}-X$
		\end{description}
		
		It follows from our assumption that up to hemimorphy, $\left(  G,H\right)  $
		is one of the following pairs: $\left(  G_{i},G_{i}\right)  $, $\left(
		\overline{G_{i}},\overline{G_{i}}\right)  $ for some $i\in\left\{
		1,2,5,6\right\}  $, $\left(  G_{3},G_{3}\right)  $, $\left(  G_{4}%
		,G_{4}\right)  $ and $\left(  F,F\right)  $ which complete the proof of
		 the implication $ii)\Longrightarrow iii)$. The implications $i)\Longrightarrow ii)$
		 and $iii)\Longrightarrow i)$ are trivial.

	\end{proof}
\begin{figure}[h]
	\begin{center}
		\setlength{\unitlength}{0.7cm}
		\begin{picture}(30,10.1)
		
		\linethickness{1mm}
		\put(0,6){\circle*{.2}}
		\put(1.5,8){\circle*{.2}}
		\put(3,6){\circle*{.2}}
		\put(1.3,5){${G_{1}}$}
		\put(6,6){\circle*{.2}}
		\put(7.5,8){\circle*{.2}}
		\put(9,6){\circle*{.2}}
		\thicklines
		\put(6,6){\vector(3,4){.9}}
		\put(6,6){\line(3,4){1.5}}
		\put(7.3,5){${G_{2}}$}
		\put(12,6){\circle*{.2}}
		\put(13.5,8){\circle*{.2}}
		\put(15,6){\circle*{.2}}
		\thicklines
		\put(12,6){\line(1,0){3}}
		\put(12,6){\vector(1,0){1.8}}
		\put(12,6){\line(3,4){1.5}}
		\put(13.5,8){\vector(-3,-4){.9}}
		\put(15,6){\line(-3,4){1.5}}
		\put(15,6){\vector(-3,4){.9}}
		\put(13.1,5){$G_{3}$}
		\put(17,6){\circle*{.2}}
		\put(18.5,8){\circle*{.2}}
		\put(20,6){\circle*{.2}}
		\thicklines
		\put(17,6){\line(1,0){3}}
		\put(20,6){\vector(-1,0){1.8}}
		\put(17,6){\line(3,4){1.5}}
		\put(18.5,8){\vector(-3,-4){.9}}
		\put(18.5,8){\line(3,-4){1.5}}
		\put(20,6){\vector(-3,4){.9}}
		\put(18.1,5){$G_{4}$}

		\thicklines
		\qbezier (1,2)(2.5,2.25)(4,2)
		\qbezier (1,2)(2.5,1.7)(4,2)
		\put(1,2){\circle*{.2}}
		\put(2.5,4){\circle*{.2}}
		\put(4,2){\circle*{.2}}
		\thicklines
		\put(2.4,2.125){\vector(1,0){0.26}}
		\put(2.6,1.85){\vector(-1,0){0.26}}
		\put(2.502,4){\line(-3,-4){1.5}}
		\put(2.5,4){\vector(-3,-4){.9}}
		\put(2.3,1){$F$}
		
		
		\put(9,2){\circle*{.2}}
		\put(10.5,4){\circle*{.2}}
		\put(12,2){\circle*{.2}}
		\thicklines
		\put(9,2){\line(1,0){3}}
		\put(9,2){\vector(1,0){1.8}}
		\put(9,2){\line(3,4){1.5}}
		\put(9,2){\vector(3,4){.9}}
		\put(10.1,1){$G_{5}$}
		
		
		\put(16,2){\circle*{.2}}
		\put(17.5,4){\circle*{.2}}
		\put(19,2){\circle*{.2}}
		\thicklines
		\put(16,2){\line(1,0){3}}
		\put(19,2){\vector(-1,0){1.8}}
		\put(16,2){\line(3,4){1.5}}
		\put(16,2){\vector(3,4){.9}}
		
		\put(17.1,1){$G_{6}$}
		
		\put(9,0){Figure 2}
		
		\end{picture}
	\end{center}
\end{figure}


	\begin{proposition}
	\label{23key}If $G$ and $H$ are $\{2,3\}$-hemimorphic flag-free digraphs then
	they have the same idiosyncratic polynomial.
	\end{proposition}
	
	\begin{proof}
	By Theorem \ref{Biltnew}, it suffices to prove that if $W$ is a module of $G$
	then $G$ and $Inv(W,G)$ have the same idiosyncratic polynomial. Let
	$V=\left\{  v_{1},\ldots,v_{n}\right\}  $ the common vertex set of $G$ and
	$Inv(W,G)$. Without loss of generality, we can assume that $W=\left\{
	v_{k+1},\ldots,v_{n}\right\}$. The generalized adjacency matrices of $G$ and
	$Inv(W,G)$ are as in Proposition \ref{HLowey} where $\beta:=\mathbbm{1}$ is the column vector of
	$1$'s and hence they have the same idiosyncratic polynomial.
	\end{proof}

	\section{Discussion and concluding remarks}
	Corollary \ref{maincas2} is trivial for 
	bipartite connected graphs. Indeed, a bipartite connected graph $G$  has exactly two transitive orientations $G^{\sigma}$
	 and $G^{\tau}$. More precisely, if  $W_1$ and $W_2$ is a bipartition of $G$, then $G^{\sigma}$
	(resp. $G^{\tau}$) 
	 is obtained by orienting all the edges of $G$ from $W_1$ to $W_2$ (resp. $W_2$ to $W_1$).
	  These orientations are known under the name 
	'\emph{canonical orientations}' \cite{anuradha2013skew}. They are the same idiosyncratic polynomial 
	because the generalized adjacency matrix of 
	one is the transpose of the other.

    Let $G$ be a bipartite graph with $n$ vertices. The adjacency matrix of $G$ has the form 
      \[
     A=\left(
     \begin{tabular}
     [c]{c|c}%
     $0$ & $U$\\\hline
     $U^{T}$ & $0$%
     \end{tabular}
     \right)
     \]
	The adjacency matrices of the canonical orientations $G^{\sigma}$
	and $G^{\tau}$ are respectively
	
	 \[
	A_1=\left(
	\begin{tabular}
	[c]{c|c}%
	$0$ & $U$\\\hline
	$0$ & $0$%
	\end{tabular}
	\right)
	\mbox{, }
	A_2=\left(
	\begin{tabular}
	[c]{c|c}%
	$0$ & $0$\\\hline
	$U^{T}$ & $0$%
	\end{tabular}
	\right)
	\]
	Consider the Seidel adjacency matrices $S_1$ and $S_2$ of $G^{\sigma}$
	and $G^{\tau}$, that is $S_1=A_1-A_1^T$ and $S_2=A_2-A_2^T$. These matrices have the same
	characteristic polynomial $Q(X)$ because $S_2=S_1^T$.
	
	Let 
	 $
	 D=\left(
	 \begin{tabular}
	 [c]{c|c}%
	 $I$ & $0$\\\hline
	 $0$ & $iI$%
	 \end{tabular}
	 \right)
	 $.
	 It is not difficult to see that $S_1=iDAD^{-1}$. Then  $Q(X)=i^nP(-iX)$
	  where $P$ is the characteristic polynomial of $A$. Two questions arise from this fact.
	  
	  \begin{question}
	  	Let $G$ be a bipartite graph and let $G^{\sigma}$ be a canonical orientation of 
	  	$G$. Is it true that the idiosyncratic polynomial of $G^{\sigma}$ can be expressed
	  	in term of the characteristic polynomial of $G$?
	  \end{question}
	  
	  \begin{question}
	  Let $G^{\sigma}$ be a transitive orientation of a comparability graph $G$.
	  We denote by $A$ the adjacency matrix of $G$ and by $S$ the Seidel adjacency 
	  matrix of $G^{\sigma}$. Is it true that the characteristic polynomial of
	  $S$ can be expressed in term of the characteristic polynomial of $A$?
	  \end{question}
	  
	  Both of these questions have a negative answer. For the first, consider  
	    the smallest pair $(G,H)$ of cospectral graphs. Let 
	     $G^{\sigma}$ (resp. $H^{\tau}$) the transitive orientation of $G$ (resp. $H$) (see figure 3).
	    The idiosyncratic polynomials of $G^{\sigma}$ and $H^{\tau}$ are respectively 
	    $(X+y) ^{3}(X^{2}-3yX-4y-4z) $ and $\left(  X+y\right)  ^{2}\left( X^{3} -3Xy^{2}-2X^{2}y-4y^{2}%
	    z-4Xy-4Xz-4y^{2}\right) $. The difference
	     between the two polynomials is $4yz(y-1)(X+y)^2$. For the second question, consider the
	      posets $P_1$ and $P_2$ drawn below. The comparability graphs of these posets 
	     have the same characteristic 
	       polynomial $X^{7}-9X^{5}-8X^{4}+8X^{3}+8X^{2}$. However, 
	       the  characteristic polynomials of the Seidel adjacency matrices 
	       of $P_1$ and $P_2$ are respectively $X^{7}+9X^{5}+8X^{3}$ 
	       and $X^{7}+9X^{5}+12X^{3}$. The second counterexample was found using SageMath.
	      
	      \begin{figure}[h]
	     	\begin{center}
	     		\setlength{\unitlength}{0.7cm}
	     		\begin{picture}(30,11.1)
	     		
	     		\linethickness{0.3mm}
	     		\thicklines

	     		\put(7.8,6.4){2}
	     		\put(4,7){\circle*{.2}}
	     		\put(8,7){\circle*{.2}}
	     		\put(8,9){\circle*{.2}}
	     		\put(4,9){\circle*{.2}}
	     		\put(3.8,6.4){1}


	     		\put(5,0){$P_2$}
	     		\put(8.3,5.8){$P_1$}
	     		\put(17.7,5.1){$H$}

	     		\thicklines
	     		\put(0,2){\circle*{.2}}
	     		\put(0,2){\vector(0,1){1.3}}
	     		\put(0,2){\vector(1,0){2.4}}
	     		\put(-0.2,1.2){2}
	     		\put(0,4){\circle*{.2}}
	     		\put(-0.2,4.2){1}
	     		\put(4,2){\circle*{.2}}
	     		
	     		\put(4,4){\circle*{.2}}
	     		
	     		\put(3.8,4.2){6}
	     		\put(4,4){\vector(0,-1){1.3}}
	     		\put(4,4){\vector(-1,0){2.4}}
	     		\put(8,2){\circle*{.2}}
	     		\put(7.8,1.4){4}
	     		\put(10,3){\circle*{.2}}
	     		\put(9.9,3.2){7}
	     		
	     		\put(3.8,1.2){3}
	     		\put(8,4){\circle*{.2}}
	     		\put(7.9,4.2){5}
	     		\put(8,2){\vector(0,1){1.3}}
	     		
	     		\put(4,4){\vector(2,-1){2.8}}
	     		\put(4,4){\vector(1,0){2.3}}
	     		\put(8,4){\vector(-2,-1){2.8}}
	     		\put(8,2){\vector(-2,0){2.3}}
	     		\put(8,7){\vector(-1,0){2.3}}
	     		\put(8,7){\vector(0,1){1.3}}
	     		\put(4,7){\vector(2,1){2.8}}
	     		\put(6,8){\circle*{.2}}
	     		\put(5.9,8.2){5}
	     		\put(11,7){\circle*{.2}}

	     		\put(20,6){\circle*{.2}}
	     		\put(19.8,5.4){2}
	     		\put(15.8,5.4){4}
	     		\put(15.8,9.2){1}
	     		\put(19.9,9.2){3}
	     		\put(18,7.7){5}
	     		\put(18,7.5){\circle*{.2}}
	     		\put(20,9){\circle*{.2}}
	     		\put(16,9){\circle*{.2}}
	     		\put(16,6){\circle*{.2}}
	     		\put(16,6){\line(1,0){4}}
	     		\put(16,9){\line(0,-1){3}}
	     		\put(16,9){\line(1,0){4}}
	     		\put(20,9){\vector(-1,0){2}}
	     		\put(20,9){\vector(0,-1){1.8}}
	     		\put(20,9){\line(0,-1){3}}
	     		\put(16,6){\line(0,1){2}}
	     		\put(16,6){\vector(1,0){2}}
	     		\put(16,6){\vector(0,1){1.8}}

	     		\put(17,3){\circle*{.2}}
	     		\put(19.5,3){\circle*{.2}}
	     		\put(17,0.7){\circle*{.2}}
	     		\put(17,3){\line(0,-1){2.4}}
	     		\put(17,3){\line(1,0){2.4}}
	     		\put(17,3){\line(1,-1){2}}
	     		\put(17,3){\line(-1,-1){2}}
	     		
	     		\put(19,1){\circle*{.2}}
	     		
	     		\put(15,1){\circle*{.2}}
	     		\put(17.8,0){$G$}
	     		\put(19.5,3){\vector(-1,0){1.4}}
	     		\put(17,0.7){\vector(0,1){1.4}}
	     		\put(19,1){\vector(-1,1){1.4}}
	     		\put(15,1){\vector(1,1){1.4}}
	     		\put(16.8,3.2){1}
	     		\put(14.7,0.4){2}
	     		\put(16.8,0.1){3}
	     		\put(18.9,0.4){4}
	     		\put(19.6,3){5}
	     		\put(8,-1.3){Figure $3$.}
	     	

	     		\put(10.8,6.4){6}
	     		\put(11,9){\circle*{.2}}
	     		\put(11,9.2){7}
	     		\put(8,7){\vector(-2,1){2.9}}
	     		\put(8,9){\vector(-2,-1){3}}
	     		\put(4,9){\vector(2,-1){3}}
	     		\put(4,9){\vector(1,0){2}}

	     		\put(3.8,9.2){4}
	     		\put(7.9,9.2){3}
	     		\put(4,9){\vector(0,-1){1.4}}

	     		\put(11,7){\line(0,1){2}}
	     		\put(11,7){\vector(0,1){1.4}}
	     		\put(8,7){\line(-2,1){4}}
	     		\put(0,2){\line(1,0){2}}
	     		\put(0,2){\line(0,1){2}}
	     		\put(2,2){\line(1,0){2}}
	     		\put(4,2){\line(1,0){4}}
	     		\put(4,2){\line(2,1){4}}
	     		\put(8,2){\line(0,1){2}}
	     		\put(8,4){\line(-1,0){4}}
	     		\put(4,7){\line(2,1){4}}
	     		
	     		\put(0,4){\line(1,0){4}}
	     		
	     		\put(4,4){\line(0,-1){2}}
	     		\put(8,2){\line(-2,1){4}}
	     		\put(4,9){\line(1,0){4}}
	     		\put(4,7){\line(0,1){2}}
	     		\put(4,7){\line(1,0){4}}
	     		\put(8,7){\line(0,1){2}}

	     		%
	     		
	     		\end{picture}
	     	\end{center}
	     \end{figure}

	\bibliography{bibpaper}
	\end{document}